\documentclass[12pt]{amsart}
\usepackage[margin=1in]{geometry}

\usepackage{graphicx}
\usepackage{amsthm}
\usepackage{amssymb}
\usepackage{xcolor}

\usepackage{enumitem}
\usepackage{hyperref}
\usepackage{varwidth}

\newtheorem{theorem}{Theorem}
\newtheorem{prop}[theorem]{Proposition}

\theoremstyle{definition}
\newtheorem{definition}[theorem]{Definition}
\newtheorem{example}[theorem]{Example}
\theoremstyle{remark}

\numberwithin{theorem}{section}
\numberwithin{equation}{section}

\newcommand{\Z}{\mathbb{Z}}

\newcommand{\mb}{\mathbb}

\renewcommand{\phi}{\varphi}

\newcommand{\paren}[1]{\left(#1\right)}
\newcommand{\set}[1]{\left\{#1\right\}}

\newcommand{\abs}[1]{\left\lvert#1\right\rvert}

\DeclareMathOperator{\rank}{rank}

\begin{document}

\title{Fuglede's conjecture fails in 4 dimensions over odd prime fields}

\author{Samuel J.~Ferguson}
\address{Courant Institute of Mathematical Sciences\\ New York University\\ 251 Mercer St.\\ New York, NY 10012, USA}
\email{ferguson@cims.nyu.edu}

\author{Nat Sothanaphan}
\address{Courant Institute of Mathematical Sciences\\ New York University\\ 251 Mercer St.\\ New York, NY 10012, USA}
\email{natsothanaphan@gmail.com}

\begin{abstract}
Fuglede's conjecture in $\mathbb{Z}_{p}^{d}$, $p$ a prime, says that a subset $E$ tiles $\mathbb{Z}_{p}^{d}$ by translation if and only if $E$ is spectral, meaning any complex-valued function $f$ on $E$ can be written as a linear combination of characters orthogonal with respect to $E$.
We disprove Fuglede's conjecture in $\Z_p^4$ for all odd primes $p$, by using log-Hadamard matrices to exhibit spectral sets of size $2p$ which do not tile, extending the result of Aten et al.~\cite{REU} that the conjecture fails in $\Z_p^4$ for primes $p \equiv 3 \pmod 4$ and in $\Z_p^5$ for all odd primes $p$.
We show, however, that our method does not extend to $\mathbb{Z}_{p}^{3}$.
We also prove the conjecture in $\Z_2^4$, resolving all cases of four-dimensional vector spaces over prime fields.
Our simple proof method does not extend to higher dimensions. The authors, however, have written a computer program to verify that the conjecture holds in $\Z_2^5$ and $\Z_2^6$.
Finally, we modify Terry Tao's \cite{Tao} counterexample to show that the conjecture fails in $\Z_2^{10}$. Fuglede's conjecture in $\mathbb{Z}_{p}^{d}$ is now resolved in all cases except when $d=3$ and $p\geq 11$, or when $p=2$ and $d=7,8,9$.
\end{abstract}

\maketitle


\section{Introduction}

In this paper, we attempt to resolve Fuglede's conjecture for all finite-dimensional vector spaces over prime fields of dimensions $d\neq 3$, and we succeed except for vector spaces over the field $\mathbb{Z}_{2}$ of dimensions $d=7,8,9$.

We briefly review Fuglede's conjecture \cite[p. 119]{Fug} in its original context of $\mathbb{R}^{d}$. Throughout this paragraph, let $E$ denote a measurable subset of $\mathbb{R}^{d}$ of finite, positive Lebesgue measure. We call $E$ a \emph{spectral set} if the Hilbert space $L^{2}(E)$ has an orthogonal basis of complex exponentials $\{x\mapsto e^{2\pi i (\lambda\cdot x)}\}_{\lambda\in\Lambda}$, for some exponent set $\Lambda\subset\mathbb{R}^{d}$. On the other hand, we say that $E$ \emph{tiles} $\mathbb{R}^{d}$ by translation if there exists a translation set $T\subset\mathbb{R}^{d}$ such that $\bigcup_{t\in T}(E+t)$ differs from $\mathbb{R}^{d}$ by a set of measure zero and, for all distinct $t, t^{\prime}\in T$, $(E+t)\cap (E+t^{\prime})$ is of measure zero. \emph{Fuglede's conjecture} states that $E$ is a spectral set if and only if $E$ tiles $\mathbb{R}^{d}$ by translation.

Fuglede originally stated his conjecture as an attempt to provide a more explicit description of his solution to a problem Segal posed in 1958 which, according to Jorgensen \cite{Jor}, arose from the work of von Neumann on the foundations of quantum mechanics. Essentially, Segal asked for a characterization of the domains $E\subset\mathbb{R}^{d}$ of finite Lebesgue measure such that there exist, on $L^{2}(E)$, commuting self-adjoint restrictions of the operators $\frac{1}{2\pi i}\frac{\partial}{\partial x_{1}},\dots, \frac{1}{2\pi i}\frac{\partial}{\partial x_{d}}$. Under a technical condition on $E$ that was later removed by Pedersen \cite{Ped}, Fuglede proved that a domain $E$ of finite measure has the above property if and only if $E$ is a spectral set, in which case each exponent set $\Lambda$ gives rise to unique commuting restrictions which have $\Lambda$ as their joint spectrum. Thus an exponent set is also called a \emph{spectrum} for $E$, and hence the term ``spectral set." Fuglede then formed the conjecture that spectral sets are precisely the sets that tile $\mathbb{R}^{d}$ by translation, a more explicit, geometric condition.

We are interested in Fuglede's conjecture in $\mathbb{Z}_{p}^{d}$, $p$ a prime. Already in his paper \cite[p. 120]{Fug}, Fuglede considered spectral and tiling sets in cyclic groups as a way of testing the likelihood of his conjecture, and pointed out that the notions of spectral and tiling sets have extensions to locally compact abelian groups. Restricting our attention to $\mathbb{Z}_{p}^{d}$, we say that a nonempty subset $E$ of $\mathbb{Z}_{p}^{d}$ \emph{tiles} if there exists a subset $T$ of $\mathbb{Z}_{p}^{d}$ such that $\bigcup_{t\in T}(E+t)=\mathbb{Z}_{p}^{d}$ and, for each distinct $t, t^{\prime}\in T$, $(E+t)\cap (E+t^{\prime})=\varnothing$. For finite abelian groups $G$, the role played by complex exponentials in $\mathbb{R}^{d}$ is performed by group homomorphisms $\phi:G\to\mathbb{C}^{\times}$. In particular, it is well known that for $G=\mathbb{Z}_{p}^{d}$, the set of all such homomorphisms is given by $\{x\mapsto e^{2\pi i (\lambda\cdot x)/p}:\lambda\in\mathbb{Z}_{p}^{d}\}$, where $\lambda\cdot x = \lambda_{1}x_{1}+\cdots +\lambda_{d}x_{d}$ for $\lambda=(\lambda_{1}, \dots, \lambda_{d})$ and $x=(x_{1},\dots, x_{d})$ in $\mathbb{Z}_{p}^{d}$. We define a nonempty subset $E$ of $\mathbb{Z}_{p}^{d}$ to be \emph{spectral} if there exists a subset $\Lambda$ of $\mathbb{Z}_{p}^{d}$ such that $\{x\mapsto e^{2\pi i (\lambda\cdot x)/p}\}_{\lambda\in \Lambda}$ forms an orthogonal basis for $L^{2}(E)$. The latter is defined to be the $|E|$-dimensional vector space of all functions $f:E\to\mathbb{C}$, equipped with the inner product $\langle f, g\rangle = \sum_{x\in E}f(x)\overline{g(x)}$ for $f,g\in L^{2}(E)$. \emph{Fuglede's conjecture} states that a nonempty subset $E$ tiles $\mathbb{Z}_{p}^{d}$ if and only if it is a spectral set.

In 2004, Tao \cite{Tao} disproved Fuglede's conjecture in $\mathbb{Z}_{p}^{d}$ for $p=2$ when $d \geq 11$ and for $p=3$ when $d \geq 5$. He then lifted his counterexamples to disprove Fuglede's conjecture in $\mathbb{R}^{d}$ for $d\geq 5$, thereby raising interest in the conjecture for $\mathbb{Z}_{p}^{d}$. Could counterexamples in $\mathbb{Z}_{p}^{d}$ be given for lower dimensions, thereby disproving Fuglede's conjecture in $\mathbb{R}^{d}$ for $d \leq 4$? Iosevich, Mayeli, and Pakianathan~\cite{IMP} partially answered this question in the negative by proving the conjecture true in $\mathbb{Z}_{p}^{d}$ for all primes $p$ when $d \leq 2$. An alternate proof was given recently by Kiss, Malikiosis, Somlai, and Vizer \cite{KMSV}, and their proof utilizes the classical result of R\'{e}dei \cite{Red}, popularized by Sz\H{o}nyi \cite{Szo}, that if a nonempty set tiles, then either the set or its translation set must be a coset of a subgroup when $d=2$. The conjecture in $\mathbb{R}^{d}$ remains unresolved for $d\leq 2$, but it has been disproved for $d=3,4$ by Farkas, Kolountzakis, Matolcsi, M\'{o}ra, and R\'{e}v\'{e}sz \cite{FarkMM, FarkR, KM1, KM2, Mat}. The counterexamples for $d=3$ use groups of the form $\mathbb{Z}_{m}^{3}$ with $m$ composite.

More recently, Aten et al.~\cite{REU} disproved Fuglede's conjecture for all odd primes $p$ when $d\geq 5$,
and for all primes $p \equiv 3 \pmod 4$ when $d=4$.
They also proved it true for $p \leq 3$ when $d=3$.
Recently, Fallon, Mayeli and Villano~\cite{FMV} extended this result by proving the conjecture for $p = 5,7$ when $d=3$.
Fallon has considered taking the topic of Fuglede's conjecture for $\mathbb{Z}_{p}^{3}$ for his doctoral dissertation, and for this reason we do not consider $\mathbb{Z}_{p}^{3}$ here.

The conjecture for $p \equiv 1 \pmod 4$ when $d=4$, $p\geq 11$ when $d=3$, and $p=2$ for $4\leq d\leq 10$, were not resolved by \cite{REU, FMV, FarkMM, FarkR, IMP, KM1, KM2, Mat, Tao}. In this paper, we modify the counterexample that Aten et al.~\cite{REU} used to disprove the conjecture
in $\Z_p^4$ for primes $p \equiv 3 \pmod 4$. We thereby obtain counterexamples for all primes $p \equiv 1 \pmod 4$.
This settles the status of the conjecture in $\Z_p^4$ for odd primes $p$, as stated below in our main result.

\begin{theorem}
\label{thm:main}
Fuglede's conjecture fails in $\Z_p^4$ for all odd primes $p$.
\end{theorem}

Nevertheless, our counterexample does not work when $p=2$.
In that case,
we prove the conjecture true for $d=4$ and give a modification of Tao's \cite{Tao} counterexample
to disprove it in the case $d=10$. We also show that the simple method of proof used for $\Z_2^4$ fails to give a proof for $\Z_2^5$. However, the authors have written a computer program to verify that the conjecture indeed holds in $\Z_2^5$ and $\Z_2^6$.
The link to the computer program is:

\begin{center}
\url{https://github.com/natso26/FugledeZ_2-d}.
\end{center}

\noindent
The program consists of two parts. A complete explanation and justification that running both parts is sufficient to verify Fuglede's conjecture for $\Z_2^5$ and $\Z_2^6$ is provided in an addendum \cite{Add}.

We remark that checking whether a set tiles or is spectral using a computer is difficult because there may be no efficient algorithm for it.
The authors of \cite{KM1} partially analyzed the computational complexity of the task and showed that a related problem is NP-complete. It is computationally infeasible for us to run the program in the case where $p=2$ and $7 \leq d \leq 9$, and so the conjecture in this case remains open.

We state our conclusions as follows.

\begin{theorem}
\label{thm:p2}
Fuglede's conjecture holds in $\Z_2^d$ when $d \leq 6$ and fails in $\Z_2^d$ when $d \geq 10$.
\end{theorem}

We summarize the status of Fuglede's conjecture in $\Z_p^d$ as follows.

\begin{itemize}
\item $d\leq 2$. True for all primes $p$.
\item $d=3$. True for $p \leq 7$ and unresolved for $p \geq 11$.
\item $4\leq d\leq 6$. True for $p=2$ and false for $p \geq 3$.
\item $7 \leq d \leq 9$. Unresolved for $p=2$ and false for $p \geq 3$.
\item $d \geq 10$. False for all primes $p$.
\end{itemize}

For $d=3$, Aten et al. \cite[Thm. 1.1(g)]{REU} proved the ``tiling implies spectral" direction of the conjecture, so only the ``spectral implies tiling" direction remains.
Our work disproves only the ``spectral implies tiling" direction of the conjecture in $\mathbb{Z}_{p}^{d}$ for $d\geq 4$ when $p$ is odd, and in $\mathbb{Z}_{2}^{d}$ for $d\geq 10$.
The ``tiling implies spectral" direction in these cases remains open.

We review the counterexample of Aten et al.~\cite{REU} in Section \ref{sec:counter} and present our counterexample
in  Section \ref{sec:newcounter}, proving Theorem \ref{thm:main}.
In Section \ref{sec:bp1}, we prove Proposition \ref{prop:alog}, which shows that our counterexample cannot be improved and shows the underlying reason why it works.
This proposition also justifies the approach taken by our computer program.
In Section \ref{sec:p2}, we prove Theorem \ref{thm:p2} and state two conjectures relevant to $\mathbb{Z}_{2}^{d}$ when $7\leq d\leq 9$.

\section{The Original Counterexample}
\label{sec:counter}

We first describe the counterexample used by Aten et al.~\cite[Sec. 8]{REU}. Let $p$ be a prime.

\begin{definition}
A vector $v$ with entries in $\Z_p$ is called \emph{equidistributed} if each of the values $0,1,\dots,p-1$
appears an equal number of times as entries of $v$.
Note that the dimension of $v$ must be a multiple of $p$ for this to occur.
\end{definition}

Aten et al.~\cite{REU} refer to an equidistributed vector as a \emph{balanced} vector.

\begin{definition}
\label{def:loghadamard}
A square matrix $\mb{L}$ with entries in $\Z_p$ is called \emph{log-Hadamard} if the difference of any two distinct rows
is an equidistributed vector.
\end{definition}

While we do not consider log-Hadamard matrices over $\mathbb{Z}_{m}$ with $m$ composite here, such log-Hadamard matrices were used by Kolountzakis and Matolcsi \cite{KM1} to disprove Fuglede's conjecture in $\mathbb{Z}_{8}^{3}$, thereby disproving it in $\mathbb{R}^{3}$.

By \cite[Thm. 4.1(g)]{REU}, $\mb{L}$ is log-Hadamard if and only if $\mb{L}^T$ is,
that is, we can replace ``rows'' by ``columns'' in Definition \ref{def:loghadamard}.
Notice that adding the vector $(1,\dots,1)$, consisting only of ones, to any row or column of a matrix
preserves the property of being log-Hadamard. When each entry of a vector is a one, we refer to it as an \emph{all-one vector}.

The following theorem relates log-Hadamard matrices to Fuglede's conjecture. It is implied by \cite[Thm. 4.7]{REU}.

\begin{theorem}
\label{thm:lhtofuglede}
Let $p$ be a prime. If there exists an $m\times m$ log-Hadamard matrix with entries in $\Z_p$ and rank at most $d$, and $m$ is not a power of $p$,
then Fuglede's conjecture fails in $\Z_p^d$.
\end{theorem}

\begin{proof}
By \cite[Thm. 4.7]{REU}, there exists a spectral subset of $\Z_p^d$ of size $m$. Since $m$ is not a power of $p$, this set cannot be a tile.
\end{proof}

We are now ready to present the counterexample of Aten et al.~\cite{REU}.
We define the vectors $v_k = (0^k,1^k,\dots,(p-1)^k)$ for $k\geq 0$, letting $0^{0}$ denote $1$, although we only use
\[
v_0=(1,1,\dots,1),\quad v_1=(0,1,\dots,p-1),\quad v_2=(0^2,1^2,\dots,(p-1)^2).
\]
For vectors $v, w\in\mathbb{Z}_{p}^{p}$, let $(v,w)\in\mathbb{Z}_{p}^{2p}$ denote the vector obtained by concatenating $v$ and $w$.

\begin{theorem}[{{\cite[Sec. 8]{REU}}}]
\label{thm:counterex}
Let $p$ be an odd prime and $n$ a nonsquare modulo $p$. Let $\mb{L}$ be the $2p \times 2p$ matrix with rows
\begin{align*}
L_k &= 2k(v_1,nv_1)-k^2(v_0,nv_0), \\
L_{k+p} &= (v_2,nv_2)-2nk(v_1,v_1)+nk^2(nv_0,v_0),
\end{align*}
where $0 \leq k \leq p-1$ and the row index starts at zero.
Then $\mb{L}$ is log-Hadamard.
\end{theorem}

Observe that $\mb{L}$, as defined in Theorem \ref{thm:counterex}, has rank at most five. Thus, by Theorem \ref{thm:lhtofuglede},
Fuglede's conjecture fails in $\Z_p^5$ for all odd primes $p$.
Moreover, if $p \equiv 3 \pmod 4$, then $-1$ is a nonsquare modulo $p$.
Taking $n=-1$ yields a matrix $\mb{L}$ with rank at most four, because this choice of $n$ makes $(v_0,nv_0)$ and $(nv_0,v_0)$ parallel.
Therefore, Fuglede's conjecture fails in $\Z_p^4$ for all primes $p \equiv 3 \pmod 4$.

We might wonder whether there is something special about primes congruent to $3$ modulo $4$ that allows us to reduce the rank of $\mb{L}$ by taking $n=-1$.
Our modified counterexample in the next section shows that this is not the case.
Indeed, we illustrate below that $-1$ being a nonsquare modulo $p$ is not crucial to the argument, and that Fuglede's conjecture still fails in $\Z_p^4$ for all primes $p \equiv 1 \pmod 4$.

\section{The Modified Counterexample}
\label{sec:newcounter}

As noted in Section \ref{sec:counter}, adding a multiple of the all-one vector $(v_0,v_0)$ to any row of
the $2p \times 2p$ matrix $\mb{L}$ preserves the property of being log-Hadamard.
Let $\alpha,\beta \in \Z_p$ be constants to be determined later.
For each $0 \leq k \leq p-1$, we add $\alpha k^2(v_0,v_0)$ to row $L_k$
and add $\beta k^2(v_0,v_0)$ to row $L_{k+p}$. The result is the log-Hadamard matrix $\mb{L}'$ with rows
\begin{align}
L'_k &= 2k(v_1,nv_1)+k^2\paren{(\alpha-1)v_0,(\alpha-n)v_0}, \label{eq:L'1} \\
L'_{k+p} &= (v_2,nv_2)-2nk(v_1,v_1)+k^2\paren{(\beta+n^2)v_0,(\beta+n)v_0}. \label{eq:L'2}
\end{align}
In order for $\mb{L}'$ to have rank at most four, we want the vectors
$$\paren{(\alpha-1)v_0,(\alpha-n)v_0} \quad \text{and} \quad \paren{(\beta+n^2)v_0,(\beta+n)v_0}$$
to be parallel. This is accomplished by setting
$(\alpha-1)(\beta+n)=(\alpha-n)(\beta+n^2)$, which simplifies to
\begin{equation}
\label{eq:alphabeta}
\beta-n\alpha+(n^2+n)=0.
\end{equation}

Notice that the rank of the constructed $\mb{L'}$ is exactly four, as we can easily check that the vectors
$(v_2,nv_2)$, $(v_1,nv_1)$, $(v_1,v_1)$ and $(av_0,bv_0)$ are linearly independent for any $a,b$ that are not both zero.
Thus, we have proved the following proposition.

\begin{prop}
\label{prop:rank4}
Let $p$ be an odd prime, $n$ a nonsquare modulo $p$, and let $\alpha$ and $\beta$ satisfy \eqref{eq:alphabeta}.
Let $\mb{L}'$ be the $2p \times 2p$ matrix with rows
\eqref{eq:L'1} and \eqref{eq:L'2}, where $0\leq k \leq p-1$ and the row index starts at zero.
Then $\mb{L}'$ is log-Hadamard of rank four.
\end{prop}

For any $n$, there exist $\alpha$ and $\beta$ that satisfy \eqref{eq:alphabeta},
although we can take $\alpha=\beta=0$ only when $n=-1$ and this is available only when $p \equiv 3\pmod 4$.
This shows that there is nothing special about $p \equiv 3 \pmod 4$ with respect to this counterexample
besides the fact that we can choose $\alpha=\beta=0$.

From Proposition \ref{prop:rank4} and Theorem \ref{thm:lhtofuglede}, we conclude that Theorem \ref{thm:main} holds.

Proposition \ref{prop:rank4} gives rise to the following explicit spectral subset $E$ of $\Z_p^4$, $p$ an odd prime, of size $2p$ with spectrum $B$:
\begin{align*}
E &= \set{(k^2, k, k, \alpha-1): k \in \Z_p} \cup \set{(nk^2, nk, k, \alpha-n): k \in \Z_p}, \\
B &= \set{(0, 2k, 0, k^2): k \in \Z_p} \cup \set{(1, 0, -2nk, nk^2): k \in \Z_p},
\end{align*}
where $\alpha \in \Z_p$ and $n$ is a non-square modulo $p$.
The value of $\beta$ above is given by Equation \eqref{eq:alphabeta}.
The above set $E$ is found by calculating a rank factorization of $\mb{L}'$ as $\mb{L}'=\mb{B}\mb{E}^T$, where $\mb{E}$ and $\mb{B}$ are of size $2p \times 4$, as in \cite{Mat}.
Then, we can take the rows of $\mb{E}$ and $\mb{B}$ to be the elements of $E$ and $B$, respectively.
As $|E|=2p$, which does not divide $p^{4}$, the spectral set $E$ does not tile $\mathbb{Z}_{p}^{4}$, giving an explicit counterexample to Fuglede's conjecture.

Note that, for fixed $n$, all values of $\alpha$ yield the same $E$ up to translation.
By translational invariance \cite[Cor. 4.3(c)]{REU}, we have essentially a single counterexample for each $n$.
Notice also that, by duality, $B$ is a spectral subset with spectrum $E$.

\section{Dephased Log-Hadamard Matrices}
\label{sec:bp1}

It is natural for us to ask whether any modification of the original counterexample $\mb{L}$, similar
to the above, can produce a log-Hadamard matrix of rank less than four. In this section we
show that this is impossible.
Our proof also explains why our counterexample in Section \ref{sec:newcounter} works.
The reason is that we are \emph{dephasing} the original counterexample $\mb{L}$.

\begin{definition}
A log-Hadamard matrix is \emph{dephased} if the entries in the first row and the first column are all zero.
\end{definition}

Observe that when $\alpha=1$ and $\beta=-n^2$, our matrix $\mb{L}'$ from Section \ref{sec:newcounter} is dephased.

\begin{definition}
We say that two log-Hadamard matrices are \emph{equivalent} if they can be transformed into one another by adding all-one vectors to rows and columns and permuting rows and columns.
\end{definition}

It is easy to see that every log-Hadamard matrix is equivalent to a dephased one. Moreover, Aten et al. \cite[Cor. 5.2]{REU} showed that, in a given equivalence class of log-Hadamard matrices, \emph{some} dephased matrix is of lowest rank.

Our Proposition \ref{prop:alog} below shows that more is true: \emph{any} dephased log-Hadamard matrix has the lowest rank in its equivalence class.
In particular, any two equivalent dephased log-Hadamard matrices have the same rank.

From Proposition \ref{prop:alog},
we see that our construction is natural. Indeed, the matrix $\mb{L}'$ of Proposition \ref{prop:rank4} is equivalent to $\mb{L}$ of Theorem \ref{thm:counterex}. Also, we recall the above observation that $\mb{L}'$ is dephased when $\alpha=1$ and $\beta=-n^2$.
Hence, this $\mb{L}'$ has the lowest rank in its equivalence class; in particular, this means that our counterexample cannot be improved. In brief, we have sharpened the counterexample of Aten et al. simply by dephasing it.

Our dephasing result has many applications.
In Section \ref{sec:p2}, guided by Proposition \ref{prop:alog}, we dephase Tao's \cite{Tao} counterexample in $\Z_2^{11}$ to obtain a counterexample in $\Z_2^{10}$, which then cannot be improved.
Finally, this proposition allows our computer program, introduced in Section \ref{sec:p2}, to compute the lowest rank of log-Hadamard matrices of a given size by checking one dephased matrix from each equivalence class.

\begin{prop}
\label{prop:alog}
Any dephased log-Hadamard  matrix over $\mathbb{Z}_{p}$ has the lowest rank among all log-Hadamard matrices equivalent to it.
\end{prop}

\begin{proof}
Let $L=[L_{i,j}]$ be an $n\times n$ dephased log-Hadamard matrix over $\mathbb{Z}_{p}$. Since $L$ is dephased, we have $L_{i,j}=0$ whenever $i=1$ or $j=1$. We wish to determine if any other equivalent log-Hadamard matrix has lower rank than $L$. However, any such equivalent matrix may be obtained, up to permutation of rows and columns, from $L$ by adding multiples of the $1\times n$ vector $(1,1,\dots, 1)$ to the rows and $(1,1,\dots, 1)^{T}$ to the columns of $L$. Letting $L^{\prime}$ be obtained from $L$ by adding multiples of $(1,1,\dots, 1)$ to the rows of $L$, and letting $L^{\prime\prime}$ be obtained from $L^{\prime}$ by adding multiples of $(1,1,\dots, 1)^{T}$ to the columns of $L^{\prime}$, we see that our conclusion will be obtained if we can prove that $\text{rank}(L^{\prime\prime})\geq \text{rank}(L)$ in all cases.

\emph{Case 1}: $L^{\prime}=[L_{i,j}^{\prime}]$ satisfies $L'_{1,j}=k$ for $1\leq j\leq n$, with $k$ nonzero modulo $p$. In this case, we can add multiples of the first row of $L'$ to the other rows of $L'$ to obtain a matrix $M$ with the same rank as $L'$ but which differs from $L$ only in its first row.
Since the first row and column of $L$ are zero, $\rank(L') = \rank(M) = \rank(L)+1$.
Then, as the span of $(1,1,\dots, 1)^{T}$ is $1$-dimensional, we have
\[
|\text{rank}(L^{\prime\prime})-\text{rank}(L^{\prime})|\leq 1.
\]
So $\text{rank}(L^{\prime\prime})\geq \text{rank}(L)$, as was to be shown.

\emph{Case 2}: $k$ is zero modulo $p$. To analyze this case, first write $L=
\begin{bmatrix}
0\\
v_{1}\\
\vdots\\
v_{n-1}
\end{bmatrix},
$
with $v_{1},\dots, v_{n-1}$ being row vectors. Write $\ell=\text{rank}(L)$. Then, without loss of generality, we may suppose that $v_{1}, \dots, v_{\ell}$ are linearly independent. Now, after adding a multiple of $(1, 1, \dots, 1)$ to each row of $L$, we arrive at
\[
L^{\prime}=\begin{bmatrix}
0\\
v_{1}+k_{1}(1,\dots, 1)\\
\vdots\\
v_{n-1}+k_{n-1}(1,\dots, 1)
\end{bmatrix}
\]
over $\mathbb{Z}_{p}$, for some $k_{1}, \dots, k_{n-1}$.
We claim that $\{v_{i}+k_{i}(1,1,\dots, 1)\}_{i=1}^{\ell}$ is a linearly independent set.
Otherwise, there exist $a_{i}$, not all $0$, with $\sum_{i}a_{i}v_{i}=-\left(\sum_{i}a_{i}k_{i}\right)(1,1,\dots, 1)$.
Since the $v_i$ start with 0, comparing the first entries of both sides gives $\sum_{i} a_{i}k_{i}=0$. Thus, $\sum_{i}a_{i}v_{i}=0$, so $\{v_{i}\}_{i=1}^{\ell}$ is a linearly dependent set, a contradiction. Thus, $\text{rank}(L^{\prime})\geq \ell = \text{rank}(L)$.

Notice that in the above argument, we have used the fact that all rows of $L$ start with 0.
By applying an analogous argument to the columns, and using the fact that all columns of $L'$ start with 0 because $k=0$, we deduce that $\text{rank}(L^{\prime\prime})\geq \text{rank}(L^{\prime})$. So $\text{rank}(L^{\prime\prime})\geq \text{rank}(L)$, as was to be shown.
\end{proof}

The further question of whether some log-Hadamard matrix of size $2p \times 2p$ and rank three can be constructed remains unresolved. However, any such construction has to fail for $p \leq 7$, because Fuglede's conjecture is true in $\Z_p^3$ in that case, by the work of Fallon, Mayeli, and Villano \cite{FMV}.

Finally, our technique does not work when $p=2$. Indeed, there is no nonsquare modulo 2.
Furthermore, it is not sufficient to demonstrate the existence of a $2p \times 2p$ log-Hadamard matrix of appropriate rank, because $2p$
is now a power of $p=2$. This implies that we must look for a matrix of size at least $6\times 6$.
However, by \cite[Cor. 5.5]{REU}, the dimensions of such a matrix must be divisible by 4, so we should look for a matrix of size at least $12 \times 12$.
Tao \cite{Tao} presented one example of a matrix of this size with rank eleven. This will be discussed in the next section.

\section{The Case $p=2$}
\label{sec:p2}

In this section, we show that Fuglede's conjecture holds in $\Z_2^4$ and fails in $\Z_2^{10}$.
For the first result, we need the following proposition. First, we define a subset $E$ of $\mathbb{Z}_{2}^{d}$ to be a \emph{graph on} $V$, for $V$ a subspace of $\mathbb{Z}_{2}^{d}$, if there is a projection operator on $\mathbb{Z}_{2}^{d}$ with image $V$ that projects $E$ bijectively onto $V$.
Aten et al. \cite{REU} call a graph on a subspace a \emph{full graph set}.

\begin{prop}
\label{prop:size4}
Any subset $E$ of $\Z_2^d$ of size $4$ is a graph on some 2-dimensional subspace $V$ of $\Z_2^d$.
\end{prop}

\begin{proof}
The property of being a graph on a subspace is invariant under translation and under invertible linear transformations.
So assume by translation that $0 \in E$. If elements of $E$ lie in a 2-dimensional subspace $V$,
then $E$ is clearly a graph on $V$.
Otherwise, by applying an invertible linear transformation, we may suppose that
\[
E = \set{0,e_1,e_2,e_3},
\]
where $e_i$ is the vector with one in the $i$th slot and zeros elsewhere.
Project $E$ bijectively onto the set $E' \subseteq \Z_2^3$ obtained by forgetting all but the first three coordinate entries of the elements of $E$.
Then, apply a further projection operator whose kernel is the span of $(1,1,1)$ and whose image is $\Z_2^2 \times \set{0}$. This operator bijectively maps $E'$ onto its image.
\end{proof}

\begin{prop}
\label{prop:z24}
Fuglede's conjecture holds in $\Z_2^4$.
\end{prop}

\begin{proof}
We apply \cite[Thm. 1.1]{REU}. Let $E \subseteq \Z_2^4$.
If $E$ is spectral or a tile, then parts (a), (b) and (c) of that theorem imply that $\abs{E}$ is a power of two.
If $\abs{E}=1,2,8,16$, then parts (d) and (e) of the same theorem show that $E$ is a graph on a subspace, which tiles and is spectral.
In the remaining case, $\abs{E}=4$, Proposition \ref{prop:size4} implies $E$ is a graph on a subspace,
and so $E$ tiles and is spectral.
\end{proof}

In contrast, it is not true that every subset $E$ of $\Z_{2}^{5}$ of size $8$ is a graph on some $3$-dimensional subspace $V$ of $\Z_{2}^{5}$. The following example demonstrates this.

\begin{example}
Consider the subset $E$ of $\Z_{2}^{5}$ given by
\[
E=\{0,e_{1}, e_{2}, e_{3}, e_{4}, e_{5}, e_{1}+e_{2}, e_{1}+e_{2}+e_{3}+e_{4}+e_{5}\}.
\]
We claim that $E$ does not tile $\Z_{2}^{5}$ by translation. Granted this, it follows by \cite[Thm. 1.1(d)]{REU} that $E$ is not a graph on a $3$-dimensional subspace $V$ of $\Z_{2}^{5}$.

As $|E|=8$, and $|\Z_{2}^{5}|=32$, we see that $E$ tiles $\Z_{2}^{5}$ if and only if there exist $t_{1}, t_{2}, t_{3}\in \Z_{2}^{5}$ such that $\Z_{2}^{5}$ is the disjoint union of $E, E+t_{1}, E+t_{2}, E+t_{3}$.
As $t_{i}\in\Z_{2}^{5}$, each $t_{i}$ must be a sum of 0, 1, 2, 3, 4, or 5 elements from the set of basis vectors $\{e_{1}, e_{2}, e_{3}, e_{4}, e_{5}\}$ of $\Z_{2}^{5}$. However, if $t_{i}$ is a sum of 0, 1, 2, 4, or 5 basis vectors, then clearly $E\cap (E+t_{i})\neq\varnothing$, so $E+t_{i}$ is not disjoint from $E$. Hence, each $t_{i}$ must be a sum of 3 basis vectors.
Moreover, as $(E+t_i) \cap (E+t_j) = \varnothing$ for $i \neq j$, we have $E \cap (E+t_j-t_i) = \varnothing$.
So, the difference of two distinct $t_i$'s must be a sum of 3 basis vectors.
But we can check that the difference of any two sums of 3 basis vectors is a sum of either 0, 2, or 4 basis vectors, a contradiction.

Our computer program can also easily check that $E$ does not tile.
\end{example}

Thus, by the above example, the method used to prove Fuglede's conjecture for $\Z_{2}^{4}$ fails to resolve it for $\Z_{2}^{5}$; nevertheless, the authors have written a computer program to verify that the conjecture indeed holds for $\Z_{2}^{5}$ and $\Z_{2}^{6}$.
The code that the authors wrote for this program can be found at

\begin{center}
\url{https://github.com/natso26/FugledeZ_2-d}.
\end{center}

\noindent
It uses a similar graph-theoretic approach to that of Siripuram et al.~\cite{Sir}, as well as the information in Sloane's online library of Hadamard matrices \cite{Slo}.

While we do not intend to explain the code here, it is perhaps worth pointing out the role of dephased log-Hadamard matrices in the program's verification that there are no counterexamples to Fuglede's conjecture in $\Z_2^5$ and $\Z_2^6$.
In brief, there are two separate computer programs, \verb|CheckSpecTile| and \verb|DephasedRank|.
The former uses a brute-force approach to check that for every subset $E$ of size 8 in $\Z_2^5$ and of size 16 in $\Z_2^6$, $E$ tiles if and only if $E$ is spectral.
The latter checks that there is no spectral set of size different from a power of two in $\Z_2^5$ and $\Z_2^6$ by computing ranks of dephased log-Hadamard matrices. Computing such ranks is, by Proposition \ref{prop:alog}, enough to conclude that there are no spectral sets of size different from a power of two, by the results of Aten et al. \cite[Thm. 4.7]{REU}. Further explanation is given in the readme files found by following the link.

Finally, we show that Fuglede's conjecture fails in $\Z_2^{10}$.
Tao's \cite{Tao} counterexample of a Hadamard matrix of size $12 \times 12$ with entries $\pm 1$
produces a corresponding log-Hadamard matrix of rank eleven.
Proposition \ref{prop:alog} suggests that we should dephase this matrix, that is, add all-one vectors to its rows and columns so that the first row and column become zero, to obtain a log-Hadamard matrix of possibly lower rank.
This procedure yields
\begin{equation}
\label{eq:dozen}
\setcounter{MaxMatrixCols}{20}
L = \begin{pmatrix}
0 & 0 & 0 & 0 & 0 & 0 & 0 & 0 & 0 & 0 & 0 & 0 \\
0 & 1 & 0 & 1 & 0 & 0 & 0 & 1 & 1 & 1 & 0 & 1 \\
0 & 1 & 1 & 0 & 1 & 0 & 0 & 0 & 1 & 1 & 1 & 0 \\
0 & 0 & 1 & 1 & 0 & 1 & 0 & 0 & 0 & 1 & 1 & 1 \\
0 & 1 & 0 & 1 & 1 & 0 & 1 & 0 & 0 & 0 & 1 & 1 \\
0 & 1 & 1 & 0 & 1 & 1 & 0 & 1 & 0 & 0 & 0 & 1 \\
0 & 1 & 1 & 1 & 0 & 1 & 1 & 0 & 1 & 0 & 0 & 0 \\
0 & 0 & 1 & 1 & 1 & 0 & 1 & 1 & 0 & 1 & 0 & 0 \\
0 & 0 & 0 & 1 & 1 & 1 & 0 & 1 & 1 & 0 & 1 & 0 \\
0 & 0 & 0 & 0 & 1 & 1 & 1 & 0 & 1 & 1 & 0 & 1 \\
0 & 1 & 0 & 0 & 0 & 1 & 1 & 1 & 0 & 1 & 1 & 0 \\
0 & 0 & 1 & 0 & 0 & 0 & 1 & 1 & 1 & 0 & 1 & 1 \\
\end{pmatrix}
\end{equation}
of rank ten. One way to verify this is by the following three lines of MATLAB code.

\begin{figure}[h]
\centering
\begin{varwidth}{\linewidth}
\begin{verbatim}
H = hadamard(12)
L = (H == -1)
gfrank(L, 2)
\end{verbatim}
\end{varwidth}
\end{figure}

\noindent
Therefore, by Theorem \ref{thm:lhtofuglede}, Fuglede's conjecture fails in $\Z_2^{10}$.
From this, Proposition \ref{prop:z24}, and the results of our computer program, we have proved Theorem \ref{thm:p2}.

According to Sloane's online library of Hadamard matrices \cite{Slo}, there is only one equivalence class of $12\times 12$ log-Hadamard matrices over $\mathbb{Z}_{2}$. Hence, by Proposition \ref{prop:alog}, no $12\times 12$ log-Hadamard matrix over $\mathbb{Z}_{2}$ has rank less than $\rank(L)=10$, with $L$ as in \eqref{eq:dozen}.
Thus, as a corollary, there is no spectral set of size $12$ in $\mathbb{Z}_{2}^{d}$ for $d\leq 9$.

Replacing $12$ by any other number $m$ which is not a power of $2$, we can run the above procedure to conclude that there is no spectral set of size $m$ in $\Z_2^d$ by checking that the rank of some dephased $m \times m$ log-Hadamard matrix over $\Z_2$ in each equivalence class is more than $d$.
This is essentially what our program \verb|DephasedRank| does.

In the authors' investigations of Fuglede's conjecture in $\Z_{2}^{d}$, they have repeatedly come across numerical evidence in favor of the following unresolved conjectures.

\begin{enumerate}
    \item For $m$ different from a power of 2, the rank of a dephased $m\times m$, $\Z_2$-valued log-Hadamard matrix $L$ depends only on $m$ and is independent of its equivalence class.
    \item For $m \geq 12$, the rank mentioned above is at least 10. 
\end{enumerate}

Notice that the first assertion generalizes Proposition \ref{prop:alog}.
If the first assertion is true, then it is possible to apply our program \verb|DephasedRank| to Sloane's online library of Hadamard matrices \cite{Slo} to verify the second assertion for $m<256$. If the second assertion is true for $m<256$, this will imply, by Aten et al.~\cite[Thm.~1.1(b),~(e)]{REU}, that any subset $E$ of $\Z_{2}^{d}$ with $d\leq 9$ is not spectral if $|E|$ is not a power of $2$. Then, the size $|E|$ of any counterexample $E$ to Fuglede's conjecture in $\Z_{2}^{d}$ for $d=7,8,9$ would be a power of $2$.
Proving Fuglede's conjecture in these settings, in turn, would resolve Fuglede's conjecture for all finite-dimensional vector spaces over prime fields of dimensions $d\neq 3$.

\section*{Acknowledgements}
The first author would like to thank Palle Jorgensen for introducing him to Fuglede's conjecture. The authors would also like to thank Thomas Fallon, Azita Mayeli, and Dominick Villano for interesting discussions pertaining to the groups $\mathbb{Z}_{p}^{d}$, and for sharing their list of references. Finally, the authors wish to thank the referees for their corrections and suggestions, which have improved the paper's readability and discussion of the literature.

\end{document}